\documentclass[a4paper,12pt]{amsart}
\usepackage{amsthm,amsmath, amssymb, amscd,mathtools, braket}
\usepackage[top=30truemm, bottom=30truemm, left=25truemm, right=25truemm]{geometry}
\usepackage{color}
\usepackage[dvipdfmx]{hyperref}%\usepackage{pxjahyper}
\numberwithin{equation}{section}
\newcounter{count}

\newcommand{\num}{\stepcounter{count}\the\value{count}}
\renewcommand{\mod}{\ \mathrm{mod}\ }

\newtheorem{theorem}{Theorem}[section]
\newtheorem{lemma}[theorem]{Lemma}
\newtheorem{corollary}[theorem]{Corollary}

\newtheorem{problem}[theorem]{Problem}

\theoremstyle{remark}
\newtheorem{remark}[theorem]{Remark}

\theoremstyle{definition}
\newtheorem{notation}[theorem]{Notation}

\begin{document}

\title[Intervals without primes near an iterated linear recurrence sequence]{Intervals without primes near an iterated\\ linear recurrence sequence}

\author[K. Saito]{Kota Saito}
\address{Kota Saito\\ Department of Mathematics\\ College of Science $\&$ Technology \\ Nihon University\\Kanda\\ Chiyoda-ku\\ Tokyo\\
101-8308\\ Japan} %\\Before address: Faculty of Pure and Applied Sciences\\ University of Tsukuba\\ 1-1-1 Tennodai\\ Tsukuba\\ Ibaraki\\ 305-8577\\ Japan}
\email{saito.kota@nihon-u.ac.jp}

\thanks{The previous address of the author is  ``Faculty of Pure and Applied Sciences, University of Tsukuba, 1-1-1 Tennodai, Tsukuba, Ibaraki, 305-8577, Japan'' }

\subjclass[2020]{11B37,	11K16, 11B50}
\keywords{Linear recurrence sequence, composite numbers, Pisot number, Salem number}

\begin{abstract}
Let $M$ be a fixed positive integer. Let  $(R_{j}(n))_{n\ge 1}$ be a linear recurrence sequence for every $j=0,1,\ldots, M$, and we set $f(n)=(R_0\circ \cdots \circ R_M)(n)$, where $(S\circ T)(n)\coloneqq S(T(n))$. In this paper, we obtain sufficient conditions on $(R_{0}(n))_{n\ge 1},\ldots, (R_{M}(n))_{n\ge 1}$  so that the intervals $(|f(n)|-c\log n, |f(n)|+c\log n)$ do not contain any prime numbers for infinitely many integers $n\ge 1$, where  $c$ is an explicit positive constant depending only on the orders of $R_0,\ldots, R_M$. %For example, let $(F(n))_{n\ge 1}$ be the Fibonacci sequence. We show that the intervals
%\[
%\biggl((F\circ F\circ F)(n)-0.0624\log n,\ (F\circ F\circ F)(n)+0.0624\log n \biggr)
%\]
%do not contain any prime numbers for infinitely many integers $n\ge 1$. 
%
As a corollary, we show that if for each $j=1,2,\ldots, M$, the sequence $(R_j(n))_{n\ge 1}$ is positive, strictly increasing, and the constant term of its characteristic polynomial is $\pm 1$, then for every Pisot or Salem number $\alpha$,  the numbers $\lfloor \alpha^{(R_1\circ \cdots \circ R_M)(n)} \rfloor $ are composite for infinitely many integers $n\ge 1$. 
\end{abstract}

\maketitle

\section{Introduction}

 In 1947, Mills \cite{Mills} constructed a real number $A$ such that $\lfloor A^{3^n}\rfloor $ is a prime number for every integer $n\ge 1$, where $\lfloor x \rfloor$ denotes the integer part of $x$.  Let $\xi$ be the least number of such numbers $A$, which is called Mills' constant (See \cite[pp.~130--131]{Finch}). The author \cite{Saito} recently showed that $\xi$ is irrational. He furthermore proved that either $\xi$ is transcendental, or $\xi^{3^m}$ is a Pisot number of degree $3$ for some positive integer $m$. We recall that a real number $\alpha$ is a \textit{Pisot number} greater than $1$ if it is an algebraic integer such that all the conjugates of $\alpha$, except for $\alpha$ itself, lie in the open unit disk of the complex plane.

Assume that $\xi$ is algebraic. Then,  $\xi^{3^m}(\eqqcolon \beta)$ is a Pisot number of degree $3$ for some $m\ge 1$. By the definition of Mills' constant, $\lfloor \beta^{3^n} \rfloor$ is a prime number for every $n\ge 1$.  Therefore, if we affirmatively solved the following problem, then $\xi$ would be transcendental.

\begin{problem}\label{Problem1}
Let $\beta$ be an arbitrary Pisot number of degree $3$. Prove \textup{(}or disprove\textup{)} that the numbers $\lfloor \beta^{3^n} \rfloor$ are composite for infinitely many integers $n$.  
\end{problem} 

For a given algebraic real number $\alpha$, there are many results on the divisibility of $\lfloor \alpha^n\rfloor$ and its variants. For example, in \cite{FormanShapiro}, Forman and Shapiro observed that $\lfloor (2+\sqrt{2})^n\rfloor$ is composite infinitely often. After that, Cass \cite{Cass} proved that the number of integers $n\in[1,x]$, such that $\lfloor \alpha^n \rfloor$ is a prime number, is bounded above by $C(\log x)^2$ for some computable constant $C>0$ if $\alpha>1$ is a unit in a quadratic field $\mathbb{Q}(D)$ and $\alpha \neq (1+\sqrt{5})/2$, where  $D$ is squarefree. Therefore, for such $\alpha$, there are infinitely many $n$ such that the numbers $\lfloor \alpha^n \rfloor$ are composite.  Dubickas \cite{Dubickas2002} extended a qualitative part of Cass's result. He showed that for every Pisot or Salem number $\alpha$, the integers $\lfloor \alpha^n \rfloor$ are composite for infinitely many integers $n$. Here, we recall that a real number $\alpha$ greater than $1$ is called a \textit{Salem number} if it is a reciprocal algebraic integer of degree at least 4 such that all the conjugates of $\alpha$, except for $\alpha$ itself and $1/\alpha$, are all of modulus $1$. Salem numbers are also of particular interest. Zaimi \cite{Zaimi} showed that for every Salem number $\alpha$ of degree $d$, the numbers $\lfloor \alpha ^n \rfloor$  are divisible by $N$ for infinitely many $n$ if and only if $N\leq 2d-3$. We refer the readers to Novikas' Ph.D. thesis \cite{Novikas} for more details.

For solving Problem~\ref{Problem1}, it seems natural to study the divisibility of $\lfloor \alpha^{R(n)} \rfloor $, where $R(n)$ is an exponentially or more rapidly growing sequence. However, one can guess that the difficulty of Problem~\ref{Problem1} is similar to a problem on the divisibility of the Fermat numbers $2^{2^n}+1$. It is known that $2^{2^n}+1$ is prime for $0\leq n\leq 4$ and is composite for $5\leq n\leq 32$ (See \cite[p.~16]{Guy}). It is a long-standing problem to prove (or disprove) the infinitude of the composite Fermat numbers. Alternatively, we aim to give a class of suitable sequences $(R(n))_{n\ge 1}$ such that the numbers $\lfloor \alpha^{R(n)} \rfloor$ are composite for infinitely many integers $n\ge 1$. \\

 We now prepare several key notions to state the results of this paper. A sequence $(R(n))_{n\geq 1}$ is called an \textit{inhomogeneous linear recurrence sequence} (\textit{ILRS}) \textit{with order} $d\ge 1$ if it is a sequence of integers and there are integers $a_{0}, \ldots, a_{d-1}, b$ with $a_0\neq 0$ such that for every integer $n\geq 1$, we have
\begin{equation}\label{eq:recurrence}
R(n+d)= a_{d-1}R(n+d-1)+a_{d-2}R(n+d-2)+\cdots + a_{0}R(n)+b.
\end{equation}  
In addition, we say that $(R(n))_{\ge 1}$ is \textit{reversible} if $a_0\in \{-1,1\}$.
\begin{remark}
In this paper, the inhomogeneous term $b$ is always a constant. 
\end{remark}

We say that $(R(n))_{n\geq 1}$ is a \textit{linear recurrence sequence} (\textit{LRS}) if $b=0$. In this case, $(R(n))_{\ge 1}$ is called \textit{non-degenerate} if no quotient of two distinct roots of its characteristic polynomial $P(x)= x^{n+d}-a_{d-1}x^{n+d-1}-\cdots - a_1x- a_0$ (which may have multiple roots) is a root of unity.   \\

Let $(R(n))_{\ge 1}$ be an unbounded LRS. Dubickas \cite[Theorem~1.1]{Dubickas2018} obtained the following interesting results:
\begin{enumerate}\renewcommand{\theenumi}{D\arabic{enumi}}
\renewcommand{\labelenumi}{(\theenumi)} 
\item \label{D1} Given a positive integer $H$ there exist a collection of (not necessarily distinct) prime numbers $p_h$, where $h=-H, \ldots, H $, and two positive integers $m$ and $L$ such that for each $h\in \{-H,\ldots,H\}$ and each integer $n\ge 1$, the number $R(Ln+m)+h$ is divisible by $p_h$;
\item \label{D2} If, in addition, $(R(n))_{n\ge 1}$ is non-degenerate then there is a constant $c>0$ depending on $d$, $a_{d-1}$, $\ldots$, $a_0$ only such that the intervals 
\[
(|R(n)|-c\log n , |R(n)|+c\log n  )
\] 
do not contain any prime numbers for infinitely many integers $n\ge 1$.
\end{enumerate} 

As a corollary\footnote{Dubickas applied \eqref{D1} and \eqref{D2} with $R(n)=\mathrm{Tr}(\alpha^n)$ in \cite{Dubickas2018}, where $\mathrm{Tr}(\beta)$ denotes the trace of $\beta$.}, Dubickas further showed that \eqref{D1} and \eqref{D2} are still true even if we replace $R(n)$ with $\lfloor \alpha^n \rfloor$ for any Pisot or Salem number $\alpha$. He gave a stronger result when $\alpha$ is a Pisot number. See \cite[Corollaries~1.3 and 1.4]{Dubickas2018} for more details. Motivated by these results and Problem~\ref{Problem1},   we obtain the following extension of \eqref{D1} and \eqref{D2}.  Here, we define $(T\circ S) (n)\coloneqq T(S(n))$.

\begin{theorem}\label{Theorem:Main1}Let $M$ be a non-negative integer, and let $(R_j(n))_{n\geq 1}$ be an ILRS with order $d_j\ge 1$ for every $j=0,1,\ldots,M$. We suppose that \begin{enumerate}
\item \label{Condition1:Thm1} $(R_j(n))_{n\geq 1}$ is reversible and positive for every $j=1,2,\ldots, M$\textup{;}
\item \label{Condition2:Thm1} $(R_1\circ \cdots \circ R_M)(n)\to \infty$ as $n\to \infty$\textup{;}
\item \label{Condition3:Thm1}$((R_0\circ R_1 \circ \cdots \circ R_M) (n) )_{n\ge 1}$ is unbounded,
\end{enumerate}
where we only suppose \eqref{Condition3:Thm1} when $M=0$. Let $H$ be an arbitrary positive integer. Then, there exists a positive integer $m_0$ such that for each $m\ge m_0$, we find  a collection of \textup{(}not necessarily distinct\textup{)} prime numbers $p_h$, where $h=-H, \ldots, H $, and a positive integer $L$ such that for each $h\in \{-H,\ldots,H\}$ and each $n\ge 0$,  the number 
\[
(R_0\circ R_1\circ \cdots \circ R_M)(Ln+m)+h 
\]
is divisible by $p_h$.
\end{theorem}

\begin{theorem}\label{Theorem:Main2}Let $M$ be a non-negative integer, and let $(R_j(n))_{n\geq 1}$ be an ILRS with order $d_j\ge 1$ for every $j=0,1,\ldots,M$. We suppose that \eqref{Condition1:Thm1}, \eqref{Condition2:Thm1}, and
\begin{enumerate}\setcounter{enumi}{2}
\renewcommand{\theenumi}{\arabic{enumi}'}
\renewcommand{\labelenumi}{(\theenumi)}
\item \label{Condition3:Thm2}$|(R_0\circ R_1 \circ\cdots \circ R_M)(n)|/\log n \to \infty$ as $n\to \infty$,
\end{enumerate}
where we only suppose \eqref{Condition3:Thm2}  when $M=0$. Then, the intervals 
\[
\biggl(|(R_0\circ R_1\circ  \cdots \circ R_M)(n)|- \delta(n) ,\ |(R_0\circ R_1 \circ \cdots \circ R_M)(n)|+ \delta(n)    \biggr)
\] 
do not contain any prime numbers for infinitely many integers $n\ge 1$, where 
\begin{equation}\label{eq:delta}
\delta(n)= \frac{\log n}{2d_0\cdots d_M} +O\left(\frac{\log n}{d_0\cdots d_M} \exp\left(-C\sqrt{\log\log n} \right)    \right)
\end{equation}
for every sufficiently large $n\ge 1$ and $C$ is an absolute positive constant.
\end{theorem}
\begin{remark}
We do not have to suppose the reversibility of $R_0(n)$ in these theorems. 
\end{remark}

Theorems~\ref{Theorem:Main1} and \ref{Theorem:Main2} with $M=0$ recover \eqref{D1} and \eqref{D2}, respectively. Indeed, it is clear that Theorem~\ref{Theorem:Main1} with $M=0$ implies \eqref{D1}. In addition, if $R$ is a non-degenerate and unbounded LRS, then in \cite[Proof of Theorem~1.1]{Dubickas2018}, Dubickas showed that there exists $A=A(R)>0$ such that
\[
|R(n)| \ge A n
\]
for every sufficiently large $n\ge 1$, which leads to \eqref{Condition3:Thm2} with $M=0$. Thus, Theorem~\ref{Theorem:Main2} with $M=0$ implies \eqref{D2}.

Dubickas \cite{Dubickas2018} did not explicitly give the constant $c$ in \eqref{D2}, but we find the explicit form of it from \eqref{eq:delta}, that is
\[
c= \frac{1}{2d_0\cdots d_M}-\epsilon.
\]
For example, let $(F(n))_{n\ge 1}$ be the Fibonacci sequence. Since $(F(n))_{n\ge 1}$ is reversible, positive, strictly increasing,  and has order $2$, Theorem~\ref{Theorem:Main2} implies that the intervals 
\[
\biggl( (F\circ F \circ F) (n)- 0.0624\log n,\   (F\circ F \circ F) (n)+ 0.0624\log n\biggr) 
\]
do not contain any prime numbers for infinitely many integers $n\ge 1$.

We note that the best known function $\delta(n)$, such that the intervals $(n-\delta(n), n +\delta(n) )$ do not contain any prime numbers for infinitely many $n$, is 
\[
\frac{c (\log n) (\log \log n )(\log\log\log \log n)}{\log\log \log n}. 
\]
This bound was given by Ford, Green, Konyagin, Maynard, and Tao \cite{FGKMT}.

By applying Theorem~\ref{Theorem:Main1},  we also obtain the following result on the divisibility of elements near the numbers $\lfloor \alpha^{(R_1\circ \cdots \circ R_M)(n)} \rfloor$.
\begin{corollary}\label{Corollary1}
Let $M$ be a fixed positive integer. For every $j=1,2,\ldots, M$, let $(R_j(n))_{n\geq 1}$ be a positive, reversible, and strictly increasing  ILRS. Let $\alpha$ be a Pisot or Salem number of degree $d_0\ge 1$. Let $H$ be an arbitrary positive integer. Then, there exists a positive integer $m_0$ such that for each $m\ge m_0$, we find a positive integer $L$ such that the numbers 
\[
\lfloor \alpha^{(R_1\circ \cdots \circ R_M)(Ln+m)} \rfloor+h 
\]
are composite for all integers $|h|\leq H$ and $n\ge 1$. Furthermore, the intervals
\[
\biggl(\lfloor \alpha^{(R_1\circ \cdots \circ R_M)(n)} \rfloor-\delta(n),  \lfloor \alpha^{(R_1\circ \cdots \circ R_M)(n)} \rfloor+\delta(n)  \biggr)
\]
do not contain any prime numbers for infnitely many $n\ge 1$, where $\delta(n)$ is a function satisfying \eqref{eq:delta}. 
\end{corollary}

For example, Corollary~\ref{Corollary1} implies that for every Pisot or Salem number $\alpha$ of degree $d$ and for every $\epsilon>0$, the intervals 
\[
\left(\lfloor \alpha^{(F\circ F)(n)} \rfloor-\left(\frac{1}{8d}-\epsilon\right)\log n, \   \lfloor \alpha^{(F\circ F)(n)} \rfloor+ \left(\frac{1}{8d}-\epsilon\right)\log n  \right)
\]
do not contain any prime numbers for infinitely many $n\ge 1$. In particular,  the numbers $\lfloor \alpha^{F(F(n))} \rfloor$  are composite for infinitely many $n\ge 1$.  

We remark that the geometric sequence $(a^n)_{n\geq 1}$ is a LRS for every $a\geq 2$, but it is not reversible. Therefore, Problem~\ref{Problem1} remains unsolved. Moreover, we do not understand the divisibility of the Fermat numbers, either. The reversibility is a strong assumption because we can smoothly transfer the discussion of $R(n)$ in \cite{Dubickas2018} to $(R_0\circ R_1 \circ \cdots \circ R_M) (n)$. We desire to remove the reversibility to achieve our goal, and so we propose the following problems.

\begin{problem}
Find a non-reversible ILRS $(R(n))_{n\geq 1}$ such that for every Pisot number $\alpha$, especially of degree $3$, the numbers $\lfloor \alpha^{R(n)} \rfloor$ are composite for infinitely many $n$.
\end{problem}

\begin{problem}
Let $(F(n))_{n\ge 1}$ be the Fibonacci sequence. Let $h$ be an arbitrary integer. Prove or disprove that 
the numbers $F(2^n)+h$ are composite for infinitely many $n$.
\end{problem}

\begin{notation}

For all integers $N$ and $q\geq 1$, there uniquely exists a pair $(u,r)$ of integers such that $N=qu+r$ and $0\leq r<q$. Then, we define $N \mod q=r$. For every $x\in \mathbb{R}$, let $\lceil x \rceil$ denote the smallest integer $n$ such that $x\leq n$, and we define $\{x\}=x-\lfloor x\rfloor$. Let $p$ denote a parameter running over the set of prime numbers.   

We say that $f(x)=g(x)+O(h(x))$ for all $x\geq x_0$ if there exists $C>0$ such that $|f(x)-g(x)|\leq Ch(x)$  for all $x\geq x_0$. 
\end{notation}
\section{Auxiliary results}

We say that a sequence $(s(n))_{n\ge n_0}$ is \textit{purely $L$-periodic} if $s(n)=s(n+L)$ for every integer $n\ge n_0$. 

\begin{lemma}\label{Lemma:Periodic}Let $(R(n))_{n\geq 1}$ be an ILRS with order $d\ge 1$. Let $a_0$ be as in \eqref{eq:recurrence}. Let $q$ be an integer greater than or equal to $2$. Then, there are two integers $s$ and $L$ such that 
\begin{equation}\label{relation:Range}
s, L\in [1,  q^{d}]
\end{equation}
 and $(R(n) \mod  q)_{n\ge s}$ is purely $L$-periodic. In addition, 
 \begin{enumerate}\renewcommand{\theenumi}{\alph{enumi}}
\renewcommand{\labelenumi}{(\theenumi)}
\item \label{item:coprime}if $a_0$ and $q$ are coprime, then $s=1$\textup{;} 
\item \label{item:prime}if $q$ is a prime number, then $s\leq  |a_0|^{d}$.
\end{enumerate}
 \end{lemma}

\begin{proof}
We observe that
\[
( R(n+d-1) \mod q, R(n+d-2) \mod q, \ldots, R(n) \mod q ) \in \{0,1,\ldots, q-1\}^d
\]  
 for every $1\leq n\leq q^d+1$. Therefore, by the pigeonhole principle, there exists a pair $(s,t)$ of integers with $1\leq s<t\leq q^d+1$ such that
\begin{equation}\label{equation:system}
\left\{ \,
    \begin{aligned}
    R(t+d-1)&\equiv R(s+d-1) \mod q,  \\
R(t+d-2)&\equiv R(s+d-2) \mod q,\\  
&\hspace{10pt} \vdots\\
R(t)  &\equiv R(s) \mod q.
    \end{aligned}
\right.
\end{equation}
Let $L=t-s$. Then, $L$ and $s$ satisfy \eqref{relation:Range}. In addition, combining \eqref{eq:recurrence} and \eqref{equation:system} 
\[
R(n)\equiv R(n+L) \mod q 
\]
for all $n\geq s$. 

It remains to show \eqref{item:coprime} and \eqref{item:prime}. For \eqref{item:coprime}, we now assume that $a_0$ and $q$ are coprime. Then, there exists $u\in \{1,2,\ldots,q-1\}$ such that $a_0 u \equiv 1 \mod q$.  By \eqref{eq:recurrence}, we obtain 
\begin{equation}\label{eq:recurrence-modq}
\begin{aligned}
R(n) &\equiv u R(n+d)-ua_{d-1} R(n+d-1)\\
&\hspace{20pt}-ua_{d-2}R(n+d-2)-\cdots-ua_1 R(n+1)-ub\quad \mod q
\end{aligned}
\end{equation}
for all $n\geq 1$. Since $R(n)\equiv R(n+L) \mod q$ for all $ n\geq s$, the congruence \eqref{eq:recurrence-modq} implies that $R(n)\equiv R(n+L) \mod q$ for every $n\geq 1$. Therefore, we obtain  \eqref{item:coprime}.

For  \eqref{item:prime}, we assume that $q$ is a prime number. If $a_0$ and $q$ are coprime, then \eqref{item:prime} is trivial from $\eqref{item:coprime}$. If $a_0$ and $q$ are not coprime, then $q\mid a_0$, which implies that 
\[
s\leq q^{d}\leq  |a_0|^{d}. \qedhere
\]
\end{proof}

\begin{lemma}\label{Lemma:Iterated-periodic}
Let $M$ be a positive integer. 
For every $j=1,2,\ldots,M$, let $(R_j(n))_{n\geq 1}$ be a reversible and positive ILRS with order $d_j\ge 1$. For every positive integer $q\ge 2$, there exists $L\in \mathbb{Z}$ such that 
\begin{equation}
1\leq L \leq  q^{d_1\cdots d_{M} }  
\end{equation}
and the sequence $((R_1\circ \cdots \circ R_M) (n) \mod q)_{n\geq 1}$ is purely $L$-periodic.
\end{lemma}

\begin{proof}
We show the lemma by induction on $M$. If $M=1$, then the lemma follows from Lemma~\ref{Lemma:Periodic} and the reversibility of $(R_1(n))_{n\geq 1}$. Suppose the lemma is true for $M=N-1$ for some $N\geq 2$. Then, let 
\[
U(n)=(R_2\circ \cdots \circ R_N)(n) 
\]
for every $n\geq 1$. Since $(R_1(n))_{n\geq 1}$ is reversible, Lemma~\ref{Lemma:Periodic} with $(R(n))_{n\geq 1}\coloneqq (R_1(n))_{n\geq 1}$ implies that there exists an integer $Q$ such that 
\[
1\leq Q \leq  q^{d_1} 
\]
and  $(R_1(n) \mod q)_{n\geq 1}$ is purely $Q$-periodic. By applying the inductive hypothesis to $U(n)$, there exists an integer $L$ such that 
\[
1\leq L \leq Q^{d_2\cdots d_M}\leq q^{d_1\cdots d_M} 
\] 
and the sequence $( U(n) \mod Q)_{n\geq 1}$ is purely $L$-periodic. Therefore, for every $n\geq 1$,  there exists an integer $j$ (which is not neseccerily positive), we have 
\[
U(n+L)= U(n)  + jQ.
\]
Since   $(R_1(n) \mod q)_{n\geq 1}$ is purely $Q$-periodic, we obtain 
\[
(R_1\circ U)(n+L)=R_1( U(n)+j Q)\equiv  (R_1\circ U) (n)   \mod q 
\]
for every $n\geq 1$. 
\end{proof}

\begin{lemma}\label{Lemma:Key}
Let $(R_0(n))_{n\geq 1}$ be an ILRS with order $d_0\ge 1$. Let $M$ and $(R_j(n))_{n\geq 1}$  $(j=1,2,\ldots ,M)$ be as in Lemma~\ref{Lemma:Iterated-periodic}. Suppose that $(R_1\circ \cdots \circ R_M)(n)\to \infty$ as $n\to \infty$. Then, there exists $m\ge 1$ depending only on $R_0,R_1, \ldots, R_M$ such that for every prime number $p$, we find an integer $L$ such that 
\[
1\leq L\leq p^{d_0d_1\cdots d_M}
\]
and the sequence $((R_0\circ R_1\circ \cdots \circ R_M) (n) \mod p)_{n\geq m}$ is purely $L$-periodic.  
\end{lemma}

\begin{proof}
Let $U(n)=(R_1\circ \cdots \circ R_M)(n)$ for every positive integer $n$. Let $a_0$ be as in \eqref{eq:recurrence} for $R\coloneqq R_0$. By applying Lemma~\ref{Lemma:Periodic} with $R\coloneqq R_0$, $a_0\coloneqq a_0$, and $q\coloneqq p$, there are two integers $s$ and $Q$ such that 
\[
 1\le s\le  |a_0|^{d_0}, \quad 1\le Q\le   p^{d_0},
\]
and the sequence $(R_0(n) \mod p)_{n\geq s}$ is purely $Q$-periodic. By the supposition on $U(n)$, there exists $m=m(R_0,\ldots, R_M)>0$ such that $U(n)\ge s$ for all $n\ge m$. Further, by Lemma~\ref{Lemma:Iterated-periodic}, there exists an integer $L$ such that 
\[
1\leq L \leq Q^{d_1\cdots d_M}\leq p^{d_0d_1\cdots d_M}  
\]
and $(U(n) \mod Q)_{n\ge m}$ is purely $L$-periodic. Therefore, for every $n\geq m$, we find an integer $j$ (which is not necessarily positive) such that $U(n+L) = U(n) + jQ$, and hence   $\min(U(n+L), U(n))\geq s$ leads to
\[
(R_0 \circ U)(n+L) =R_0 (U(n+L) ) = R_0 (U(n)+ j Q  )\equiv (R_0 \circ U) (n) \mod p
\]
for every $n\geq m$.
\end{proof}

\begin{lemma}\label{Lemma:Trace} Let $\alpha$ be a non-zero algebraic integer of degree $d\ge 1$. Let $\alpha_1, \alpha_2, \ldots, \alpha_{d}$ be all the conjugates of $\alpha$. Let $R(n)=\alpha_1^{n}+\cdots+\alpha_{d}^{n}$ for every $n\geq 1$. Then, $(R(n))_{n\geq 1}$ is an ILRS.
\end{lemma}
\begin{proof}It is clear that $R(n)\in \mathbb{Z}$ for every $n\ge 1$ from the theory of fields. Let $f(X)=X^{d}-a_{d-1} X^{d-1} -a_{d-2}X^{d-2} -\cdots -a_0$ be the minimal polynomial of $\alpha$ with integral coefficients. Since $\alpha_j^{d} -a_{d-1} \alpha_j^{d-1}-\cdots - a_0=f(\alpha_j)=0$ for each $j=1,2,\ldots d$, we obtain 
\begin{align*}
\alpha_j^{n+d} &= \alpha_j^{n} (a_{d-1}\alpha_j^{d-1} +a_{d-2} \alpha_j^{d-2}+\cdots + a_0 )\\
& = a_{d}\alpha_j^{n+d-1}+a_{d-2}\alpha_j^{n+d-2}+\cdots + a_0\alpha_j^n.
\end{align*}
Therefore, by taking the summation over $j$, we have
\[
R(n+d)=a_{d-1}R(n+d-1)+a_{d-2}R(n+d-2)+\cdots + a_0R(n). 
\] 
for every $n\geq 1$.
\end{proof}

\begin{theorem}\label{Theorem:PNT}
There exists $C>0$ such that for every real number $X\ge 2$, 
\begin{align*}
\vartheta(X)\coloneqq \sum_{p\leq X} \log p &= X+O\left(X\exp(-C\sqrt{\log X} )  \right).
\end{align*}
\end{theorem}
\begin{proof}
See \cite[Theorem~6.9]{MontgomeryVaughan}.
\end{proof}

\section{proof of the main theorems and Corollary}
\begin{proof}[Proof of Theorem~\ref{Theorem:Main1}]Fix an arbitrary integer $H$.  Let $m_0$ be a sufficiently large integer depending only on $M$, $R_0$, $R_1$, $\ldots$, $R_M$, and $H$. If necessary, we implicitly replace $m_0$ with a larger one. We take an arbitrary integer $m\ge m_0$.  Let $U(n)=(R_1 \circ \cdots \circ R_M)(n)$ for each $n\geq 1$, where $U(n)=n$ when $M=0$. 

By \eqref{Condition3:Thm1} in Theorem~\ref{Theorem:Main1}, we may suppose that $|(R_0\circ U)(m)|-H\ge 2$. 
Take an arbitrary integer $h$ such that $|h|\leq H$. Let $p_h$ be the smallest prime factor\footnote{The prime numbers $p_h$ $(|h|\leq H)$ are not necessarily distinct.} of $(R_0 \circ U)(m)+h$.  By \eqref{Condition1:Thm1} and \eqref{Condition2:Thm1}, applying Lemma~\ref{Lemma:Key} with $p\coloneqq p_h$, there exists a positive integer $L(p_h)$ such that the sequence 
\[
((R_0\circ U)(n) \mod p_h)_{n\ge m} 
\]
is purely $L(p_h)$-periodic. Let 
\[
L= \prod_{p\in \{p_h \colon |h|\leq H \}} L(p).  
\] 
Therefore,  for every $n\ge 0$, we have 
\[
(R_0\circ U)(L(p_h)n+m)+h \equiv (R_0\circ U)(m)+h  \equiv 0 \mod p_{h}.\qedhere
\]
\end{proof}
\begin{proof}[Proof of Theorem~\ref{Theorem:Main2}]
Let $m_0$ be a sufficiently large integer depending only on $M$, $R_0$, $R_1$, $\ldots$, $R_M$,  and take an arbitrary integer $m\ge m_0$. Let $H=H(m)=|(R_0\circ U)(m)|-2$. By \eqref{Condition3:Thm2} in Theorem~\ref{Theorem:Main2}, we have  
\[
H= |(R_0\circ U)(m)|-2 \geq 2\log m.
\]
Let $p_h$ be the smallest prime factor of $(R_0\circ U)(m)+h$ for every $|h|\leq H$. Since $m\ge m_0$ and $m_0$ is sufficiently large, by \eqref{Condition1:Thm1} and \eqref{Condition2:Thm1}, Lemma~\ref{Lemma:Key} implies that for every $|h|\leq H$, there exists  $L(p_h)\in \mathbb{Z}$  such that $((R_0 \circ U)(n) \mod p_h)_{n\ge m}$ is purely $L(p_h)$-periodic and 
\[
1\leq L(p_h)\leq p_h^{D},
\]
where $D=d_0\cdots d_M$.  Let 
\[
L=\prod_{p\in \{p_h\colon |h|\leq H  \} } L(p).   
\]
Since the definition of $H$ yields that 
\[
\biggl[|(R_0\circ U)(m)|-H,\ |(R_0\circ U)(m)|+H\biggr]=[2,2H+2]
\]
and $p_h$ is the smallest prime factor of $(R_0\circ U)(m)+h$, we obtain 
\[
\{p_h\colon |h|\leq H \}=\{p \colon p\leq 2H+2 \}. 
\]
Thus, we have
\[
1\leq L \leq  \left(\prod_{p\leq 2H+2 }   p^D \right) = \left( e^{\sum_{p\leq 2H+2 } \log p } \right)^D= e^{D\vartheta (2H+2) }.
\]
Let $n=\lceil e^{D\vartheta (2H+2)}/L \rceil $ and take an arbitrary integer $h$ with $|h|\leq H$. Then, we have 
\begin{align*}
|(R_0\circ U) (Ln+m)|+h &= \pm (R_0\circ U) (Ln+m)+h\\
&\equiv \pm((R_0\circ U) (m)\pm h ) \equiv 0 \mod p_{\pm h}.
\end{align*}
Therefore, the number $|(R_0\circ U) (Ln+m)|+h$ is divisible by $p_h$ or $p_{-h}$. Further, by \eqref{Condition3:Thm2} and Theorem~\ref{Theorem:PNT}, we obtain  
\begin{align*}
|(R_0\circ U) (Ln+m)|+h&\geq (2/D)\log (Ln+m)-H\\
&\geq (2/D)\log e^{D\vartheta(2H+2)} -H>2H+2.
\end{align*}
Therefore, $|R_0\circ U (Ln+m)|+h$ is a composite number.

Let $N=Ln+m$. Then, combining the above estimates and Theorem~\ref{Theorem:PNT} implies that
\[
 \log N=D\vartheta(2H+2) +O(1)=   2HD + O(HD \exp(-C\sqrt{\log H} )   ) , 
 \]
and hence
\[
H=  \frac{\log N}{2D} +O\left(\frac{\log N}{D} \exp\left(-\frac{C}{2}\sqrt{\log\log N} \right)    \right)=\delta(N).
\]
 By choosing $m=m_0, m_0+1,\ldots$, the intervals
 \[
\biggl (|(R_0 \circ U)(N)|-\delta(N),\ |(R_0 \circ U)(N)|+\delta(N) \biggr)
 \]
do not contain any prime numbers for infinitely many integers $N\ge 1$. 
\end{proof}

\begin{proof}[Proof of Corollary~\ref{Corollary1}]Let $U(n)=(R_1\circ \cdots \circ R_M)(n)$ for every $n\geq 1$. The sequences $(R_1(n))_{j\ge 1},\ldots, (R_M(n))_{j\ge 1}$ clearly satisfy \eqref{Condition1:Thm1} and \eqref{Condition2:Thm1} since they are reversible, positive, and strictly increasing.

Let $\alpha$ be a Pisot or Salem number of degree $d_0\ge 1$. Let $\alpha_1=\alpha,\alpha_2,\ldots, \alpha_{d}$ be all the conjugates of $\alpha$. Let $R(n)$ be as in Lemma~\ref{Lemma:Trace}. Then, for every $n\geq 1$, we have
\[
\lfloor \alpha_1^{U(n)} \rfloor =\alpha_1^{U(n)} -\{\alpha_1^{U(n)} \}=(R\circ U) (n) +g(n),  
\]
where $\{x\}$ denotes the fractional part of $x$ and $g(n)\coloneqq -\alpha_2^{U(n)}\cdots - \alpha_{d}^{U(n)}-\{\alpha_1^{U(n)} \}$. 

If $\alpha$ is a Pisot number, then all of the conjugates of $\alpha$ except for $\alpha$ itself lie in the unit open disk. Thus, $g(n)$ is bounded and $((R\circ R_1\circ \cdots\circ R_M) (n))_{n\ge 1}$ satisfies \eqref{Condition3:Thm1} and \eqref{Condition3:Thm2} since $(U(n))_{n\ge 1}$ is strictly increasing and $\alpha=\alpha_1>1$. Therefore,  Theorems~\ref{Theorem:Main1} and \ref{Theorem:Main2} imply the corollary for Pisot numbers.

If $\alpha$ is a Salem number, then $d$ should be even and greater than or equal to $4$, and all of the conjugates of $\alpha$ are formed as $\alpha$, $1/\alpha$, $e^{2\pi i \theta_1}$, $e^{-2\pi i \theta_1}$, $\ldots$, $e^{2\pi i \theta_r}$, $e^{-2\pi i \theta_r}$, where $r=(d-2)/2$. Therefore, $g(n)$ is also bounded and $((R\circ R_1\circ \cdots\circ R_M)  (n))_{n\ge 1}$ satisfies \eqref{Condition3:Thm1} and \eqref{Condition3:Thm2}. Theorems~\ref{Theorem:Main1} and \ref{Theorem:Main2} lead to the corollary for Salem numbers.
\end{proof}

\section*{Acknowledgement}
The author was supported by JSPS KAKENHI Grant Number JP25K17223.

\bibliographystyle{amsalpha}
\bibliography{references}

\end{document}